\documentclass[12pt]{amsart}
\usepackage[colorlinks=true,pagebackref,hyperindex,citecolor=Green,linkcolor=Blue]{hyperref}
\usepackage[top=1in, bottom=1in, left=1.1in, right=1.1in]{geometry}
\usepackage{amsmath}
\usepackage{amsfonts}
\usepackage{amssymb}
\usepackage{amsthm}
\usepackage{stmaryrd}
\usepackage[all]{xy}  
\usepackage{mathrsfs}
\usepackage{enumitem}
\usepackage[T1]{fontenc}

\usepackage{color}

\makeatletter
\def\namedlabel#1#2{\begingroup
#2%
\def\@currentlabel{#2}%
\phantomsection\label{#1}\endgroup
}
\makeatother

\theoremstyle{theorem}
\newtheorem{theorem}{Theorem}[section]

\newtheorem{corollary}[theorem]{Corollary}
\newtheorem{lemma}[theorem]{Lemma}
\newtheorem{proposition}[theorem]{Proposition}

\newtheorem{theoremx}{Theorem}

\theoremstyle{definition}
\newtheorem{definition}[theorem]{Definition}
\newtheorem{definition-theorem}[theorem]{Definition/Theorem}

\newtheorem{example}[theorem]{Example}
\newtheorem{remark}[theorem]{Remark}
\numberwithin{equation}{subsection}

\newcommand{\m}{\mathfrak{m}}

\newcommand{\Hom}{\operatorname{Hom}}

\newcommand{\lHom}{\operatorname{Hom}^l}

\newcommand{\op}{\operatorname{op}}

\newcommand{\cx}{\operatorname{cx}}

\newcommand{\dd}{\operatorname{d}}
\newcommand{\reg}{\operatorname{reg}}
\newcommand{\Tor}{\operatorname{Tor}}
\newcommand{\h}{\operatorname{H}}
\newcommand{\tr}{\operatorname{Tr}}

\newcommand{\pre}{^{e}}
\newcommand{\excise}[1]{}

\newcommand{\ds}{\displaystyle}

\renewcommand{\a}{\mathfrak{a}}

\definecolor{blue-violet}{rgb}{0.54, 0.17, 0.89}
\definecolor{Blue}{rgb}{0.01, 0.28, 1.0}
\definecolor{gGreen}{rgb}{0.2, 0.8, 0.2}
\definecolor{Green}{rgb}{0.04, 0.85, 0.32}

\begin{document}

\title{The Frobenius Exponent of Cartier Subalgebras}

\author{Florian Enescu}
\author{Felipe P\'erez}
\address{Department of Mathematics and Statistics, Georgia State University, 14th floor, 25 Park Place, Atlanta, GA 30303}
\email{fenescu@gsu.edu, felipe.perez@zerogravitylabs.ca}
\maketitle

\begin{abstract}
Let $R$ be a standard graded finitely generated algebra over an $F$-finite field of prime characteristic, localized at its maximal homogeneous ideal.  In this note, we prove that that Frobenius complexity of $R$ is finite. Moreover, we extend this result to Cartier subalgebras of $R$.
\end{abstract}

\section{Introduction}
Let $R$ be a local ring of positive characteristic $p>0$. A Cartier operator of order $e$ on $R$, also known as {$p^{-e}$-linear map (or operator), is an additive map $\phi:R\to R$ such that $\phi(r^{p^e}s)=r\phi(s)$ for all $r,s \in R$. These maps have been widely studied in the recent years because of their connections to  tight closure theory in commutative algebra and their applications to the study of singularities in algebraic geometry. See for example  \cite{TestQGor,BST-F-signature-of-pairs,BlickleTestIdealsViaAlgebrasOfP-eLinearMaps} or \cite{BlickleP-1LinearMaps} for an excellent survey.

Let $\mathcal{C}^R_e$ denote the set of $p^{-e}$-linear maps on $R$. For a $p^{-e}$ linear map $\phi$ and a $p^{-e'}$-linear map $\psi$, the map $\phi \psi$ is $p^{-(e+e')}$-linear. In particular the abelian group $ \mathcal{C}^R=  \oplus_{e\geq 0} \mathcal{C}^R_e$ has a natural structure of graded ring, with multiplication given by composition of maps. A \emph{Cartier subalgebra on $R$} is a graded subring $\mathcal{D}\subseteq \mathcal{C}^R$, such that $\mathcal{D}_0=R$ and $\mathcal{D}_e \neq 0$, for some $e >0$. Cartier subalgebras appear naturally in the contexts of ideal pairs and divisor pairs \cite{H-Y,HaraWatanabe}, concepts widely used in algebraic geometry.

The Cartier subalgebras on $R$ are rarely commutative or finitely generated which make them difficult to study. One is interested in understanding the ``size'' of the Cartier subalgebras. In \cite{BST-F-signature-of-pairs} this is done by considering the growth of  the number of splitting in each degree $e$. This led to the definition of the $F$-signature of the pair $(R,\mathcal{D})$. The fact, shown in loc. cit., that the $F$-signature of the pair is finite implies that the number of splitting is understood asymptotically. In this note, we are interested in another invariant that measures the abundance of Cartier maps, that is the complexity of the Cartier subring $(R,\mathcal{D})$. We will introduce this concept in this paper along the lines of the Frobenius complexity for a local ring defined in~\cite{EnescuYao}.

Let $A$ be a graded ring, not necessarily commutative, and let $X$ be a set of homogeneous elements in $A$ that minimally generates $A$.  The complexity of $A$ is the growth rate of the elements in $X$ of degree $e$, as $e$ varies (see \cite{EnescuYao} or Section \ref{sec:Background}). This invariant was introduced by the first author and Yao in loc. cit., where the authors studied the complexity of the ring of Frobenius operators on the injective hull of a local ring $(R,\m)$. The complexity appears as a new invariant associated to singularities and there are still many open questions related to it. One such important question regards its finiteness. In this note, we establish the finiteness of this number in the graded setting. 

\begin{theoremx}[Theorem \ref{thm: main}]
Let $R$ be a finitely generated standard graded $k$-algebra of positive characteristic localized at its homogeneous maximal ideal. Then the Frobenius complexity of $R$ is finite.
\end{theoremx}

As mentioned earlier, we extend the notion of Frobenius complexity to Cartier subalgebras. We define the Frobenius exponent for a Cartier subalgebra $(R,\mathcal{D})$, and we show that, for Cartier subalgebras of bounded generating degree growth, the Frobenius exponent of $(R,\mathcal{D})$ is finite, see Definition~\ref{bounded} and Theorem~\ref{thm: main}. Both of our results, Theorem \ref{thm: main} and Corollary \ref{cor: main}, are effective. 

In Section 2 of the paper, we recall and collect the basic definitions and results we need, while the main results are in Section 3.

\section{Background} \label{sec:Background}

For a ring of prime characteristic $p >0$, there is a natural map, $F:R\to R$, namely the Frobenius endomorphism, given by $r\mapsto r^p$. Similarly, for $e\geq 0$, we have the iterated Frobenius endomorphism $F^e:R\to R$. We say that $R$ is is $F$-finite when $F$ is a finite ring homomorphism. Otherwise stated, we assume $R$ to be $F$-finite.

For any $R$-module $M$ the $e$th iterated Frobenius endomorphism induces a second $R$-module structure on $M$ via scalar restriction. Let $\pre M$ denote the $R$-$R$-bimodule whose element set is $M$, and has $R$-$R$-bimodule structure given by 
\[r\cdot m\cdot s=  r^qsm\]
for $r,s$ in  $R$ and $m$ in $\pre M$. Equivalently, the left $R$-module structure of  $\pre M$ is given via the $e$the iterated Frobenius map on $R$, meanwhile the right one is the usual action of $R$ in $M$. We define $M^{e}$ in a similar fashion, with $R$-$R$-bimodule structure given by \[r\cdot m\cdot s=rs^qm\] for $r,s\in R$ and $m\in M^{e}$.

\subsection{The total Cartier algebra}

For some time now, there has been interest in the $p^{-e}$-linear maps on $R$, since they carry a great deal of information about the singularities of the ring $R$ in prime characteristic $p$. We recall the basic definitions and results related to these maps, following ~\cite{BST-F-signature-of-pairs} as our main reference, while noting that some of these definitions were introduced earlier in ~\cite{TestQGor}.

A $p^{-e}$-linear map on $R$ is by definition an $R$-linear map $^{e}R \to R$. We denote the collection of $R$-linear maps from $^{e}R$ to  $R$ by $\mathcal{C}^R_e$, that is  \[\mathcal{C}^R_e=\lHom_R\left({^{e}R}, R\right).\]
Observe that the abelian group \[ \mathcal{C}^R :=\bigoplus_{e\geq 0} \mathcal{C}_e^R=\bigoplus_{e\geq 0} \lHom_R(^{e}R,R)\] has a natural structure of a non-commutative graded ring, where multiplication is given by composition; given $\phi:{^{e}R}\to R$ and $\psi:{^{e'}R}\to R$, we let $\phi\psi={^{e'}\phi}\circ \psi:{^{e+e'}R}\to R$. This ring is called the {\it total Cartier algebra} on $R$.

\subsection{The ring of Frobenius operators} A dual concept to that of a Cartier map (or Cartier operator) is that of a Frobenius operator. Let $(R,\m)$ be a local commutative ring of positive characteristic $p$, and $M$ be an $R$-module. The module of Frobenius operators of order $e$-th on $M$ is defined as the $R$-module \[\mathcal{F}^e(M)=\Hom_R(M,{^{e}M}).\] In an analogous way to the total Cartier algebra on $R$, the abelian group $ \mathcal{F}(M)=\oplus_{i=0}\mathcal{F}^i(M)$ has a natural graded ring structure, with multiplication given by composition. The ring $\mathcal{F}(M)$ is  called \emph{the ring of Frobenius operators on $M$}.  

The Frobenius operators of order $e$ satisfy a Fedder's type of isomorphism.

\begin{proposition}[{\cite[Proposition 3.6]{KatzmanSchwedeSinghZhang},\cite[Proposition 3.36]{B-IntersectionHomologyDModulesInFiniteChar}}] \label{prop: Frob operators equal colon}
Let $R$ be a complete local ring such that $R=S/I$ where $S$ is a complete, regular, local ring of positive characteristic. If $E=E_R(k)$ denotes the injective hull of the residue field of $R$, then there is an isomorphism of left $R$-modules \[\mathcal{F}^e(E)\cong \frac{(I^{[p^e]}:I)}{I^{[p^e]}}.\]

\end{proposition}

When Proposition \ref{prop: Frob operators equal colon} is combined with Fedder's isomorphism (\cite[Lemma 1.6]{FedderFputityFsing}) one gets 

\begin{corollary}[see~\cite{AlvBoixZar}] \label{cor:frob comple = comple cartier}
Let $R$ be an $F$-finite complete local ring. There is an isomorphism of $R$-$R$-bimodules 
\[ \mathcal{F}^e(E)=\Hom_R(E,{^{e}E}) \cong \left(\Hom_R({^{e}R,R})\right)^{\op}=(\mathcal{C}_e^R)^{\op},\] where $(-)^{\op}$ represents that the bimodule structure is reversed. 

\end{corollary}

\subsection{The description of $p^{-e}$-linear maps in our setting}  Let $k$ be an $F$-finite field, $S=k[x_1,\ldots,x_n]_{(x_1,\ldots,x_n)}$ and $I$ an ideal in $S$. It is known that there exists a generator of $\Hom_S(^{e}S, S)$ as an $^{e}S$-module (via premultiplication), by \cite[Lemma 1.6]{FedderFputityFsing}, so let us denote this map by $\tr$, by analogy to case when $k$ is perfect.

In the case $k$ is perfect, one can take this generator to be the so called trace map, customarily denoted by $\tr$, which can be described as follows:

Assume $k$ is perfect. Let $S=k[x_1,\ldots,x_n]_{(x_1,\ldots,x_n)}$, then the $S$-module ${^eS}$ is a free module with basis consisting of the monomials $x_1^{a_1}\cdots x_n^{a_n}$, with $0\leq a_i\leq p^e-1$ for all $i$. Furthermore, the $S$-linear map $\tr:{^eS}\to S$ that sends $x_1^{p^e-1}\cdots x_n^{p^e-1}$ to $1$, and all other basis elements to $0$, called the trace map, is a generator for $\Hom_S({^eS},S)$ (via premultiplication).

Returning to case $k$ is $F$-finite, let $R=S/I$, where $I$ is an ideal of $S$. Via Fedder's result, we have the following $R$-module isomorphism   \[\Psi:\frac{(I^{[p^e]}:I)}{I^{[p^e]}}\to \Hom_R({^{e}}R,R),\] 
defined  by
\[\Psi(\overline{s}) = \overline{\tr(s\cdot -)},\] for $\overline{s} \in \frac{(I^{[p^e]}:I)}{I^{[p^e]}}$. Note that, on the right, $\overline{\cdot}$ means going modulo $I$ and the $R$-module structure on the module is given by premultiplication. In particular, for any ideal $\mathfrak{a}\subset S$, we get an isomorphism  \[\Psi:\frac{\mathfrak{a}(I^{[p^e]}:I)+I^{[p^e]}}{I^{[p^e]}}\to \overline{\mathfrak{a}}\Hom_R({^{e}}R,R).\]

\begin{definition}
A \emph{Cartier subalgebra on $R$} is a graded $k$-subalgebra $\mathcal{D}=\oplus_{e\geq0}\mathcal{D}_e$ of $\mathcal{C}^R$ satisfying $\mathcal{D}_0=\mathcal{C}_0=\Hom_R(R,R)=R$, and $\mathcal{D}_e \neq 0$ for some $e>0$.  
\end{definition}

We bring to the attention of the reader that the setting of Cartier subalgebras allows an uniform treatment of pairs and triples as discussed in Section 4.3.1 of ~\cite{BST-F-signature-of-pairs}.

\begin{definition} {\cite{BlickleTestIdealsViaAlgebrasOfP-eLinearMaps}}
A $F$-graded system of ideals of $R$ is a sequence $(\mathfrak{a}_e)_{e\geq 0}$ of ideals of $R$ such that $a_0=R$ and 
\[\mathfrak{a}_e\mathfrak{a}_{e'}^{[p^e]}\subseteq a_{e+e'},\]
for all $e,e'\geq 0$.  
\end{definition}

The set of Cartier operators of degree $e$, $\Hom_R({^eR},R)$, has a second $R$-module structure given by premultiplication by elements of $R$. More precisely, given $r\in R$ and $\phi:{^eR}\to R$ we set for $(r\cdot \phi)(s)=\phi(rs)$. Given a $F$-graded system of ideals $(\mathfrak{a}_e)_{e\geq 0}$, there is a Cartier subalgebra \[\mathcal{C}_{\mathfrak{a}_{\bullet}}=\bigoplus_{e\geq 0}\mathfrak{a}_e\cdot \Hom_R({^eR},R).\] Indeed, if $r_e\in \mathfrak{a}_e$, $r_{e'}\in \mathfrak{a}_{e'}$ , $\phi \in \Hom_R({^eR},R)$, and $\phi' \in \Hom_R({^{e'}R},R)$ then, \[(r_e\phi_e) * (r_{e'}\phi_{e'})= (\phi_e \circ F_*^{e}\phi_{e'})\left(r_er_{e'}^{p^e}\cdot - \right)=r_er_{e'}^{p^e}(\phi_e \circ F_*^{e}\phi_{e'}) \in \mathfrak{a}_{e+e'}\Hom_{e+e'}({^{e+e'}R},R).\] 

In fact, when $R$ is a Gorenstein ring there is an equivalence between Cartier subalgebras and $F$-graded systems of ideals \cite[Lemma 4.9]{BST-F-signature-of-pairs}.

\begin{remark} 
Let $\mathfrak{a}$ an ideal of $R$ and $t$ a positive real number. A typical example of an $F$-graded system of ideals is $\mathfrak{a}_e=\mathfrak{a}^{\lceil t p^e-1 \rceil}$.
\end{remark}

\subsection{Complexity} Let $A=\oplus_{i\geq 0} A_i$ be a graded ring, not necessarily commutative. We define 
\begin{itemize}
\item  $G_e(A)=G_e$ the subring of $A$ generated by degrees less than or equal than $e$. 
\item $k_e(A)=k_e$ the minimal number of homogeneous generators of $G_e$ . 
\end{itemize} 
The quantity $k_e-k_{e-1}$ is then the minimal number of elements of degree $e$ that are needed in a homogeneous generating set for the ring. We call the sequence $(k_e-k_{e-1})_e$ \emph{the complexity sequence} of $A$. The complexity of a ring measures the growth rate of the complexity sequence. More precisely,
we have the following definition.

\begin{definition}
The \emph{complexity of $A$} is defined as \[\cx(A)=\inf\{n\in \mathbb{R}_{>0}:k_e-k_{e-1}=\mathcal{O}(n^e)\}.\]
\end{definition}

We can recover the complexity sequence as the minimal number of generators of a quotient of $A_e$.

\begin{proposition}[{\cite[Corollary 2.4]{EnescuYao}}] \label{prop: c_e = generators of quotient} 
The minimal number of generators of $\frac{A_e}{(G_{e-1})_e}$ as an $A_0$-bimodule is $k_e-k_{e-1}$.
\end{proposition}

In this note we are interested to understand the complexity of Cartier subalgebras and the ring of Frobenius operators. As the growth of these rings is expected to be dependent on the characteristic, we give the following definition.

\begin{definition}
[C.f. \cite{EnescuYao}]
\label{exp} 
Let $R$ be a local ring of positive characteristic $p$, and let $A$ be a Cartier subalgebra. We define \emph{the Frobenius exponent of $A$} as \[\exp_F(A)=\log_p\left(\cx(A)\right).\]

\end{definition}

In the case that $A=\mathcal{C^R}$ is the total Cartier algebra of a complete local ring $(R,\m,k)$, we see that 
\[\cx_F(R)=\cx_F(\mathcal{F}(E_R(k)))=\exp_F(\mathcal{C^R}).\]

In fact, one can note that if $(R, \m, k)$ is local, then $E_R(k) =E _{\hat{R}}(k)$. It can be checked that 
$\mathcal{F}(E_R(k)) \otimes_R \hat{R}= \mathcal{F}(E_{\hat{R}}(k))$ and $\mathcal{C^R} \otimes_R \hat{R} = \mathcal{C^{\hat{R}}}.$ This shows that the complexity sequences of $\mathcal{F}(E_R(k))$ and $\mathcal{C^R}$ remain invariant under completion and hence their computation can be done after taking first the completion, if necessary.

\section{Results - The graded case}

In this section we prove our main results, namely Theorem ~\ref{thm: main} and its corollary.

We fix some notation, throughout this section $T=k[x_1,\ldots,x_n]$ is the polynomial ring in $n$ variables, $J\subseteq T$ is a homogeneous ideal, and $\m=(x_1,\ldots,x_n)$.  Take $S=T_{(x_1,\ldots,x_n)}$, $I$ the image of $J$ in $S$, and $R=S/I$. Additionally, we assume that $[k : k^p] < \infty$ and hence $R$ is $F$-finite. We also let $(R,\mathcal{D})$ be a Cartier subalgebra,  and denote by $J_e^\mathcal{D}$ the unique ideal of $S$ such that $J_e^\mathcal{D}\supseteq I^{[p^e]}$ and $\Psi\left( \frac{J_e^{\mathcal{D}}}{I^{[p^e]}} \right)= \mathcal{D}_e$, where $\Psi$ is the map defined in Section 2.3.

\begin{remark}
For an $F$-graded system of ideals $(\a_e)_{e\geq 0}$ and the Cartier subalgebra associated to it, $\mathcal{D}=\mathcal{C}_{\a_{\bullet}}$, one has $J_e^{\mathcal{D}} = \a_e(I^{[p^e]}\colon I)+I^{[p^e]}$.

\end{remark}

\begin{definition}
Let $T=k[x_1,\ldots,x_n]$ with the usual grading. Given a homogeneous ideal $J\subseteq S$, we define the \emph{generating degree of $J$}, denoted by $\dd(J)$, as the maximal degree among a minimal set of generators. Equivalently, \[ \dd(J) =\inf\{m \in \mathbb{Z}\mid J=J_{\leq m}T\}\]
\end{definition}

\begin{remark} \label{rmk: bound on regularity}
By \cite[Theorem 16.3.1]{BroSharp} and \cite[Corollary 2.7]{CharacteristicFreeBoundsForTheCastelnuovoMunfordRegularityCS} we have \[\dd(J)\leq \reg(J) \leq (2\dd(J))^{2^{\dim(T)-2}},\]where $\reg(J)$ denotes the regularity of an ideal $J$.

\end{remark}

\begin{definition}
\label{bounded}
We say that an $F$-graded system of ideals $(\mathfrak{a}_e)_{e \geq 0}$ on $T$ has  \emph{bounded generating degree growth} $n$ if each $\mathfrak{a}_e$ is homogeneous, and $\dd(\mathfrak{a}_e)=\mathcal{O}(n^e)$ for some number $n$.  
For example, if $\mathfrak{a}$ is a homogeneous ideal of $T$ and $t$ is a positive real number, the family $\mathfrak{a}_e=\mathfrak{a}^{\lceil t (p^e-1)\rceil}$ is an $F$-graded system of ideal with bounded generating degree growth. Indeed, $\dd(\mathfrak{a}^{\lceil t (p^e-1)\rceil})=\dd(a)\lceil t (p^e-1)\rceil= \mathcal{O}(p^e)$. 

Furthermore, a Cartier subalgebra $\mathcal{D}$ is called of  \emph{bounded generating degree growth} if $\{J_e^{\mathcal{D}}\} _e$ has bounded generating degree growth.
\end{definition}

Note that if $(\mathfrak{a}_e)_{ e\geq 0}$ is a $F$-graded system of ideals on $T$, then $(\mathfrak{a}_eR)_{e \geq 0}$ is a $F$-graded system of ideals on $R$. We need the following lemma:

\begin{lemma} \label{lem: degree bound}
Let $J\subseteq T=k[x_1,\ldots,x_n]$ be a homogeneous ideal and  $d=\dd(J)$. If $g\in J$ is homogeneous polynomial of degree larger than $d$, then $g\in \m J$. Equivalently,  \[\m^{t}\cap J \subseteq \m J\]for all $t>d$.
\end{lemma}
\begin{proof}
Let $g_1,\ldots,g_v$ be a minimal set of homogeneous generators for $J$. If $g$ is a homogeneous polynomial of degree larger than $d$, and $g\notin \m J$, then we can find homogeneous elements $r_i $ for which
\[g=r_1g_1+\ldots + r_vg_v\] 
with some $r_i\notin \m$. But this implies that $g$ and $g_i$ have the same degree, which is smaller than or equal to $d$ a contradiction. 
\end{proof}

\begin{lemma} \label{lem: D_el correspondence} Let $R=S/I$. If $\mathcal{D}_{e,l} = \{\phi \in \mathcal{D}_e \mid \phi(^eR)\subseteq \m^lR\}$, then
\[\Psi\left(\frac{(\m^{l})^{[p^e]}\cap J_e^{\mathcal{D}}+I^{[p^e]}}{I^{[p^e]}}\right) = \mathcal{D}_{e,l}.\]
\end{lemma}

\begin{proof}

As discussed above, $s\mapsto \tr(s\cdot -)$ induces the isomorphism $\Psi$ that, when restricted to $J_e^\mathcal{D}/I^{[p^e]}$, gives an isomorphism with $\mathcal{D}_e$. Let $s\in (\m^{l})^{[p^e]}\cap J_e^{\mathcal{D}}+I^{[p^e]}$, then we can write $s=s'+r$ with $s'\in (\m^l)^{[p^e]}$ and $r\in I^{[p^e]}$. We deduce 
\[\tr(s\cdot(^eS) )\subseteq \tr(s'\cdot (^eS))+\tr(r\cdot (^eS))\subseteq \m^l+I,\]
that is $\Psi(\bar{s})=\Psi(\bar{s'})\in \mathcal{D}_{e,l}$. This shows $``\subseteq"$.

To prove the other inclusion, suppose that $s\in J_e^{\mathcal{D}}\subseteq(I^{[p^e]}\colon I)$ is such that $\Psi(\bar{s})\in \mathcal{D}_{e,l}$, then $\tr(s\cdot (^eS))\subseteq \m^{l}+I$. Thus $\phi(s)\in \m^l+I$ for all $\phi \in \Hom_S(^eS,S)$, this implies $s \in (\m^l+I)^{[p^e]}=(\m^l)^{[p^e]}+I^{[p^e]}$, that is $s=s'+r$ with $s'\in (\m^l)^{[p^e]}$ and $r\in I^{[p^e]}$. But $I^{[p^e]} \subseteq J_e^{\mathcal{D}}$, so $\bar{s}=\bar{s'}\in \frac{(\m^{l})^{[p^e]}\cap J_e^{\mathcal{D}}+I^{[p^e]}}{I^{[p^e]}}$, which in turn implies the result.
\end{proof}

\begin{proposition} \label{prop:c_e < l}
Let $A=\oplus_{e \geq 0}A_e$ be a graded ring, then $A_e$ has a natural structure of left and right $A_0$-module. If $\mu_{A_0}^l$ and $\mu_{A_o}^r$ denote the minimum number of generators as left and right $A_0$-modules, then  \[k_e(A)-k_{e-1}(A)\leq \min\{ \mu_{A_0}^l(A_e),\mu_{A_0}^r(A_e)\}\]  
\end{proposition}

\begin{proof}
By Proposition ~\ref{prop: c_e = generators of quotient} above  $k_e-k_{e-1}$ is the minimal number of generators of $\ds \frac{A_e}{(G_{e-1})_e}$ as an $A_0$-$A_0$-bimodule. The results now follows. 
\end{proof}

\begin{lemma} \label{lem: hom of finite module length bound}
Let $M$ be a finite length $R$-module supported at $\m$. If $e>0$, then \[\ell_R\left( \Hom_R({^eR},M)\right)\leq \ell_R\left(\frac{R}{\m^{[p^e]}R}\right)\ell_R(M). \] 
\end{lemma} 
\begin{proof}
We proceed by induction on the length of $M$. If $\ell_R(M)=1$ then $M\cong R/\m R$ and so 
\[\ell_R\left( \Hom_R({^eR},M)\right)=\ell_R\left( \Hom_R\left({^eR},\frac{R}{\m R}\right)\right)\leq\ell_R\left(\frac{R}{\m^{[p^e]}R}\right).\] 
Let $m\in M$ a nonzero element in the socle of $M$, so the map $R/\m\to M$ that maps $1\mapsto m$ is injective. We have an exact sequence
\[0\to R/\m \to M \to N\to 0,\]
with $\ell_R(N)+1= \ell_R(M)$. We apply $\Hom({^eR},-)$ to the above sequence to obtain
\[\Hom({^eR},R/\m R)\to \Hom({^eR},M) \to \Hom({^eR},N).\] 
It follows that 
\begin{align*}
\ell_R \left( \Hom_R({^eR},M)\right) &\leq \ell_R\left( \Hom_R({^eR},R/\m)\right)+\ell_R\left( \Hom_R({^eR},N) \right)\\ 
&\leq \ell_R\left(\frac{R}{\m^{[p^e]}R}\right) +\ell_R(N)\ell_R\left(\frac{R}{\m^{[p^e]}R}\right) \\
&= \ell_R\left(\frac{R}{\m^{[p^e]}R}\right)\ell_R(M).
\end{align*}
\end{proof}

\begin{theorem} \label{thm: main1}
If $d_e=\dd(J_e^{\mathcal{D}}) = \mathcal{O}(p^{te})$ for some fixed $t$, then \[\exp_F(\mathcal{D})\leq t\dim(R). \] 
\end{theorem}

\begin{proof}
Let $d=\dim(R)$. For every $e$ we let $l_e=\lceil d_e/p^e \rceil+1$, Lemma \ref{lem: degree bound} implies 
\[(\m^{l_e})^{[p^e]}\cap J_e^{\mathcal{D}} \subseteq \m J_e^{\mathcal{D}}.\] 
Hence, 
\[\mu_R\left(\frac{J_e^{\mathcal{D}}}{I^{[p^e]}}\right) = \ell_R \left(\frac{J_e^{\mathcal{D}}}{\m J_e^{\mathcal{D}}+I^{[p^e]}}\right) \leq \ell_R \left( \frac{J_e^{\mathcal{D}}}{(\m^{l_e})^{[p^e]}\cap J_e^{\mathcal{D}}+I^{[p^e]}} \right).\]
From Lemma \ref{lem: D_el correspondence}, we have 
\[\frac{\mathcal{D}_e}{\mathcal{D}_{e,l_e}}\cong \frac{J_e^{\mathcal{D}}}{(\m^{l_e})^{[p^e]}\cap J_e^{\mathcal{D}}+I^{[p^e]}},\] 
which, by using the exact sequence 
\[0\to \mathcal{D}_{e,t} \to \mathcal{D}_e \to \lHom_R\left(^eR,R/\m^{l_e}R\right),\] 
gives an inclusion 
\[ \frac{\mathcal{D}_e}{\mathcal{D}_{e,l_e}} \hookrightarrow \lHom_R\left(^eR,R/\m^{l_e}R\right)\]
which, by Lemma \ref{lem: hom of finite module length bound}, implies  
\[\mu_R\left(\frac{J_e^{\mathcal{D}}}{I^{[p^e]}}\right) \leq \ell_R \left( \lHom_R(^eR,R/\m^{l_e}R) \right) \leq \ell_R\left( \frac{R}{\m^{l_e}R}\right) \ell_R \left( \frac{R}{\m^{[p^e]}R} \right)\] 
and so 
\[\mu_R\left(\frac{J_e^{\mathcal{D}}}{I^{[p^e]}}\right)\leq \mathcal{O}(l_e^d)\mathcal{O}(p^{ed})=\mathcal{O}\left( \left\lceil \frac{d_e}{p^e} \right\rceil^d p^{ed} \right)=\mathcal{O}(d_e^d)\leq \mathcal{O}(p^{tde}).\]

To finish the proof, simply note that $\mu_R\left(\frac{J_e^{\mathcal{D}}}{I^{[p^e]}}\right)= \mu^r_{\mathcal{D}_0} (\mathcal{D}_e)$ and we can apply Proposition~\ref{prop:c_e < l} to obtain 
$$\cx_F (\mathcal{D}) \leq p^{td},$$which furthermore gives $$\exp_F(\mathcal{D}) \leq t \cdot d.$$

\end{proof}

Our goal now is to give conditions for which $\dd(J_e^{\mathcal{D}})\leq \mathcal{O}(p^{te})$ for some $t$. We show that this happens among other cases when $\mathcal{D}=C_{\mathfrak{a}_{\bullet}}$ and the family of ideals $\mathfrak{a}_{\bullet}$ has bounded generating degree growth. We need to first bound the generating degree of the colon ideals $(I^{[p^e]}\colon I)$ and we follow the strategy of Katzman and Zhang \cite{CastelMumRegulDiscrefFjumpingCoeffGradedRings}. 

\begin{proposition} \label{prop: reg of quotient reg colon}
Let $I=(g_1,\ldots,g_v) \subseteq S$ be a homogeneous ideal. Suppose there is $l\geq 1$ such that \[\reg \left((I^{[p^e]}+g_iR)\right)=\mathcal{O}(p^{el})\] for all $1\leq i \leq v$, then \[\reg\left((I^{[p^e]}\colon I)\right)=\mathcal{O}(p^{el}).\]
\end{proposition}

\begin{proof}
This is an immediate consequence of the proof of Lemma 3.2 and Theorem 3.3  in \cite{CastelMumRegulDiscrefFjumpingCoeffGradedRings}, where $\mathcal{O}(p^{el})$ is used instead of $\mathcal{O}(p^{e})$.
\end{proof}

\begin{corollary}
For a homogeneous ideal $I\subseteq S=k[x_1,\ldots,x_n]$ we have  \[ \reg((I^{[p^e]}\colon I))= \mathcal{O}(p^{2^{n-2}e})\]
\end{corollary}

\begin{proof}
Let $d$ be the generating degree of $I$, note that $\dd(I^{[p^e]}+g_iS)\leq dp^e$. By Remark \ref{rmk: bound on regularity} \[\reg\left(I^{[p^e]}+g_iS\right) \leq (2dp^e)^{2^{n-2}} = C p^{2^{n-2}e}.\]
Thus \[\reg \left((I^{[p^e]}+g_iS)\right)= \mathcal{O}(p^{2^{n-2}e}).\] 
 and Proposition \ref{prop: reg of quotient reg colon} implies 
  \[ \reg ((I^{[p^e]}\colon I)) = \mathcal{O}(p^{2^{n-2}e})\]
\end{proof}

\begin{theorem} \label{thm: main}
Let $R=S/I$. Then for any $F$-graded homogeneous family of ideals $(\a_e)_{e\geq 0}$ of bounded generating degree growth $p^l$, we have \[\exp_F(\mathcal{C}_{\a_{\bullet}})\leq \max\{l,2^{n-2}\}\dim(R). \] 
\end{theorem}

\begin{proof}
Recall that \[J_e^{\mathcal{D}} = \a_e(I^{[p^e]}\colon I)+I^{[p^e]},\] and so 
\begin{align*}\dd(J_e^{\mathcal{D}}) &\leq \dd(\mathfrak{a}_e)+ \dd\left((I^{[p^e]}\colon I)\right)\\ & = \mathcal{O}(p^{le})+\mathcal{O}(p^{e(2^{n-2})}) \\ &= \mathcal{O}(p^{\max\{l,2^{n-2}\}e}) 
\end{align*} 
The result is now immediate via Theorem \ref{thm: main1}.
\end{proof}

\begin{corollary} \label{cor: main}
Let $S$ be a localization of the polynomial ring $T$ over field at its homogeneous maximal ideal and $I\subseteq T$ a homogeneous ideal.  If $R=S/I$, then $R$ has finite Frobenius complexity. 
\end{corollary}

\begin{proof}
This is a consequence of Theorem \ref{thm: main} by taking $\mathcal{C}_{a_\bullet}=\mathcal{C}^{R}$ and noting that the Frobenius complexity is invariant under completion; see also Definition~\ref{exp}.
\end{proof}

\begin{remark}
It is expected that $\reg((I^{[p^e]}\colon I))=\mathcal{O}(p^{e})$; this would imply that the Frobenius complexity of a ring is less than its dimension. In \cite[Corollary 3.6]{CastelMumRegulDiscrefFjumpingCoeffGradedRings} it is shown that this is the case if $\dim\left(\text{Sing}(R/g_iR)\cap \text{V}(I)\right)\leq 1$ for all $i$. 
\end{remark}

\begin{example}
This result is far from optimal. Consider for example $R=\mathbb{F}_p[x_1,\ldots,x_n]$. In this case the splitting dimension is $n$, but the Frobenius complexity is $-\infty$. 
\end{example}

\begin{example}
It is shown in \cite{EnescuYao} that, if $S_3$ is the completion of the determinantal ring obtained by the $2\times 2$ minors of a $3\times 2$ matrix at the homogenous maximal ideal, then \[\cx_p(S_3)=1+\log_p(p+1)-\log_p 2.\] Note that $\dim(S_3)=3$.
\end{example}

\bibliographystyle{alpha}
\bibliography{References}

\end{document}